\documentclass[11pt]{amsart}

\usepackage{amsthm,amsmath,amsfonts,amssymb,mathrsfs}

\newcommand{\inv}{^{-1}}
\newcommand{\til}{\widetilde}
\newcommand{\ov}{\overline}
\newtheorem{theorem}{Theorem}[section]
\newtheorem{proposition}[theorem]{Proposition}
\newtheorem{lemma}[theorem]{Lemma}
{\theoremstyle{definition}
}
{\theoremstyle{remark}
}

{\theoremstyle{remark}
\newtheorem{step}{Step}}

\title[The generalized word problem in graphs of groups]{An automata theoretic approach to the generalized word problem in graphs of groups}
\author{Markus Lohrey\and Benjamin Steinberg}
\thanks{The authors would like to
acknowledge the support of DFG Mercator program.  The second author is also supported by an NSERC grant.}
\address{Universit\"at
Leipzig, Institut f\"ur Informatik, Germany
 \and
School of Mathematics and Statistics,
Carleton University, ON, Canada}
\email{lohrey@informatik.uni-leipzig.de,
bsteinbg@math.carleton.ca}
\date{May 26, 2009}

\begin{document}

\begin{abstract}
We give a simpler proof using automata theory of a recent result of Kapovich, Weidmann and
Myasnikov according to which so-called benign graphs of groups preserve
decidability of the generalized word problem.  These include graphs of groups in which edge groups are polycyclic-by-finite and vertex groups are either locally quasiconvex hyperbolic or polycyclic-by-finite and so in particular chordal graph groups (right-angled Artin groups).
\end{abstract}
\maketitle

\section{Introduction}

The {\em generalized word problem} is one the classical decision problems in
group theory. For a finitely generated (f.g.) group $G$, the generalized word problem
for $G$ asks given as input elements $g, g_1,\ldots,g_n$ (represented by words over
some given generating set of $G$) whether $g$ belongs to the subgroup
generated by $g_1, \ldots, g_n$. 
Examples of groups with decidable generalized word problem are
f.g.~free groups (see for instance~\cite{Stal83}), polycyclic groups~\cite{AvWi89,Mal83},
and f.g.~metabelian groups~\cite{Rom74,Rom80}. Moreover, every subgroup separable 
finitely presented group has a decidable generalized word problem. Mikhailova~\cite{Mih59} proved
that if the generalized word problem is decidable in $G_1$ and $G_2$
then the same holds for the free product $G_1 * G_2$.
On the other hand, Mikhailova also proved that the direct product
of two free groups of rank $2$ has an undecidable generalized word
problem~\cite{Mih66}. The same was shown by Rips~\cite{Rip82b}
for certain hyperbolic groups, see~\cite{Wise03} for refinements of Rip's
construction.  Free solvable groups of rank $2$ and derived length at least $3$ also have undecidable generalized word problem~\cite{Umi95}.  It should be noted that in these undecidability results
the f.g.~subgroup for which membership is being asked is fixed, i.e., a 
fixed f.g.~subgroup $H$ of the ambient group $G$ is constructed such that
it is undecidable whether a given element of $G$ belongs to $H$.
On the other hand, all the decidability results mentioned above
are {\em uniform} in the sense that the f.g.~subgroup is part of the input.
In order to make this distinction clear, we will often use the term ``uniform
generalized word problem'' in the sequel.

The starting point of our work is a recent result of Kapovich, Weidmann and
Myasnikov~\cite{KaWeMy05}, which provides a condition for a graph of groups $\mathbb{G}$ that implies decidability of the uniform generalized word problem for the 
fundamental group $\pi_1(\mathbb{G})$. These conditions are quite technical
(see the definition of a benign graph of groups in 
Section~\ref{GWP}); let us just mention that (of course), for 
every vertex group of $\mathbb{G}$, the uniform generalized word problem
has to be decidable and that every edge group has be Noetherian (i.e.,
does not contain an infinite ascending chain of subgroups).   In~\cite{KaWeMy05} it is shown that graphs of groups in which edge groups are polycyclic-by-finite and vertex groups are either locally quasiconvex hyperbolic or polycyclic-by-finite --- and so in particular the graphs of groups representing chordal graph groups (right-angled Artin groups) --- satisfy the necessary conditions.

The proof in~\cite{KaWeMy05} uses an extension of the Stallings folding
technique~\cite{Stal83}.  The idea is to create a ``folded'' graph that recognizes a normal form for each element of the subgroup.  The need to be able to accept each element of the subgroup with one graph is what makes the folding moves in~\cite{KaWeMy05} very technical.   Our proof in Section~\ref{GWP}
is based on an automaton saturation process in the style of 
Benois construction~\cite{Benois69,BenSak86} for rational subsets of free groups, 
see also~\cite{KaSiSt06,LohSte08}. A crucial idea is that instead of looking for membership 
of a given group element $g$ in a f.g.~subgroup $H$, we check membership of
$1$ in the coset $Hg^{-1}$. This makes the whole algorithm simpler
because the normal form theorem 
for fundamental groups of graphs of groups is simplest for elements representing $1$. The ascending chain condition
for  edge groups is what guarantees that our saturation process eventually 
terminates.

\section{Cosets}
Our approach to the generalized word problem is to consider cosets of finitely
generated subgroups, rather than finitely generated subgroups.  This has the 
advantage that the uniform problem reduces to checking whether the identity 
belongs to a coset, which is often easier.

It turns out to be useful to describe cosets in a way that does not refer 
to which subgroup it is a coset of and whether it is a right or left coset.  
Such a way was considered by Schein~\cite{Schein66}.

A \emph{coset} of a group $G$ is a subset $A$ of $G$  such that $AA\inv A=A$.  
Equivalently, if we view $G$ as a universal algebra with a ternary operation 
$(x,y,z)\mapsto xy\inv z$, then a coset is a subalgebra of $G$.
Traditionally, 
cosets are required to be non-empty, but it turns out for 
this paper that it is convenient to also allow the empty set to be a coset.    
Notice that this definition of a coset is left-right dual.    
One can verify that a non-empty set $A$ is a coset in this 
sense if and only if $A=Hg$ for some subgroup $H\leq G$ and some element $g\in G$.

\begin{proposition}
Let $A$ be a non-empty coset of $G$.  Then $H=AA\inv$ is a 
subgroup of $G$ and if $g\in A$, then $A=Hg$.  Conversely, if $H$ is 
a subgroup of $G$ and $g\in G$, then $Hg$ is a non-empty coset of $G$.
\end{proposition}
\begin{proof}
Assume first that $A$ is a non-empty coset.
Since $AA\inv$ is non-empty, $AA\inv AA\inv =AA\inv$ and 
$(AA\inv)\inv = AA\inv$, it follows that $H = AA\inv$ is a 
subgroup of $G$.  Let $g\in A$.  Clearly $Hg\subseteq AA\inv A=A$.  
Conversely, if $a\in A$ then $ag\inv \in AA\inv =H$ and so $a\in Hg$.

Clearly, if $H$ is a subgroup of $G$ and $g\in G$, then 
$(Hg)(Hg)\inv Hg= H^3g=Hg$ and so $Hg$ is a non-empty coset.
\end{proof}

The set $K(G)$ of all cosets of $G$ is a complete lattice with respect to the inclusion ordering since it is the set of all subalgebras of $G$ with respect to the ternary operation considered above.  The maximum element of $K(G)$ is the coset $G$ itself.  The minimal non-empty elements are the singleton subsets of $G$, which we identify with the elements of $G$ notationally, i.e., we write $g$ instead of $\{g\}$.  Given any subset $X$ of $G$, there is a least coset $A$ containing $X$, denoted $\ov X$, and called the coset generated by $X$.   It can be described as the intersection of all cosets containing $X$. If $X=\emptyset$, then $\ov X=\emptyset$ and otherwise $\ov X$ is non-empty.   We remark that if $A$ and $B$ are two cosets containing $g$ and $A=Hg$, $B=Kg$, then $A\cap B = (H\cap K)g$.  Note that if $A$ is a coset of $G$ and $H$ is a subgroup of $G$, then $A\cap H$ is a coset of $H$ (possibly empty).

Notice that if $A\subseteq B$ are cosets and $AA\inv = BB\inv$, then $A=B$.
This is clear if $A=\emptyset$.  Otherwise, $A,B$ are non-empty. Suppose $g\in
A\subseteq B$.  Then $B=BB\inv g=AA\inv g=A$.  Thus $G$ is Noetherian (i.e.,
$G$ satisfies the
ascending chain condition on subgroups, or equivalently all its 
subgroups are finitely generated) if and only if it satisfies the ascending chain condition on cosets.

Let us say that a coset $A$ is \emph{finitely generated} if there is a finite set of elements $X\subseteq G$ so that $A=\ov X$. That is $A$ is finitely generated as an algebra with respect to the ternary operation $(x,y,z)\mapsto xy\inv z$.    The next proposition shows that a non-empty coset is finitely generated if and only if it is a coset of a finitely generated subgroup; the empty coset is of course finitely generated.

\begin{proposition}\label{generators}
Let $A$ be a non-empty coset of $G$.  Then $A$ is finitely generated if and
only if $H=AA\inv$ is a finitely generated subgroup of $G$.  
More specifically, let $g \in A$. If
$\{g_i\mid i\in I\}$ generates $H$ as a subgroup, then $\{g_ig\mid i\in
I\}\cup \{g\}$ generates the coset $A$ 
and if $\{a_i\mid i\in J\}$ generates the coset $A$, then $\{a_ig\inv\mid i\in J\}$ generates $H$ as a subgroup.
\end{proposition}
\begin{proof}
%We prove the second statement as the it implies the first.  
Assume first that $\{g_i\mid i\in I\}$ generates $H$ as a subgroup.  Then
clearly $g_ig\in Hg=A$ for all $i \in I$ and $g\in Hg=A$.  Now if $B$ is any coset
containing all $g_ig$ and $g$, then $BB\inv$ contains the $g_i$ and hence $H$.
Thus $A=Hg\subseteq BB\inv B=B$.  This shows that $A$ is generated by
$\{g_ig\mid i\in I\}\cup \{g\}$.  Next suppose that $\{a_i\mid i\in J\}$
generates $A$ as a coset.  Then clearly, $H$ contains the $a_ig\inv$.  If $K$
is any subgroup containing the $a_ig\inv$, 
then $Kg$ contains the $a_i$ and hence contains $A$.  
Thus $H=AA\inv \subseteq Kg(Kg)\inv = K$.  This shows that the $a_ig\inv$ generate $H$ as a subgroup.
\end{proof}

Notice that when going from coset generators to group generators, we can choose $g$ to be one of the $a_i$.

\section{Dual automata and cosets}
In this section we consider an automaton model for recognizing cosets of groups.    If $\Sigma$ is a set we use $\til \Sigma$ for $\Sigma$ together with a set of formal inverses $\Sigma^{-1}$.  Then $\til \Sigma^*$  denotes the free monoid on $\til \Sigma$, which we view as the free monoid with involution in the natural way.   The free group on $\Sigma$ will be denoted $F(\Sigma)$.  When convenient, we will identify $F(\Sigma)$ with the set of reduced words in $\til \Sigma^*$ and the canonical projection $\rho\colon \til \Sigma^*\to F(\Sigma)$ will often be thought of as freely reducing a word.  Throughout this article, $\Sigma$ will be assumed finite.

\subsection{Dual automata}
By a graph $\Gamma$, we mean a graph in the sense of Serre~\cite{Serre03}.  So
$\Gamma$ consists of a set $V$ of vertices, $E$ of edges, a function
$\alpha\colon E\to V$ selecting the initial vertex of an edge and a
fixed-point-free involution on $E$ written $e\mapsto e\inv$.  This involution
extends to paths in the natural way.  One defines the terminal vertex function
$\omega\colon E\to V$ by $\omega(e)=\alpha(\ov e)$.  A \emph{dual automaton}
$\mathscr A$ over $\Sigma$ 
is a $4$-tuple $(\Gamma,\iota,\tau,\delta)$ where:
\begin{itemize}
 \item $\Gamma=(V,E)$ is a graph;
\item  $\iota,\tau$ are distinguished vertices of $\Gamma$, called the
  \emph{initial} 
and \emph{terminal} vertices of $\Gamma$ respectively;
\item  $\delta\colon E\to \til \Sigma^*$ is an involution preserving map,
  i.e., $\delta(e^{-1}) = \delta(e)\inv$.
\end{itemize}
A dual automaton will be called \emph{literal} if $\delta\colon E\to \til \Sigma$.

The map   $\delta$ extends to paths in the obvious way.
The \emph{language of $\mathscr A$}, denoted $L(\mathscr A)$, 
is the subset of $F(\Sigma)$ consisting of all elements $w$ 
so that there is a path $p$ from $\iota$ to $\tau$ with $\delta(p)=w$ in $F(\Sigma)$, i.e., $\rho\delta(p)=w$.
We say that $p$ is an \emph{accepting path} for $w$.

\begin{proposition}\label{dualautomatalanguages}
The language of a (finite) dual automaton over $\Sigma$ is a (finitely generated) coset of $F(\Sigma)$.  Conversely, every (finitely generated) coset of $F(\Sigma)$ is the language of a literal (finite) dual automaton over $\Sigma$.
\end{proposition}
\begin{proof}
Let $\mathscr A=(\Gamma,\iota,\tau,\delta)$ be a dual automaton over $\Sigma$.
Let $L=L(\mathscr A)$. If $L$ is empty, then we are done, so assume it is
non-empty. It is always true that $L\subseteq LL\inv L$.  Conversely, if $w\in
LL\inv L$ with $w=uv\inv z$ such that $u,v,z\in L$ and if $p,q,r$ are
paths accepting $u,v,z$ respectively, then $pq\inv r$ accepts $w=uv\inv z$.
Thus $LL\inv L=L$ and so $L$ is a coset.  Notice that the map $\delta\colon E\to \til \Sigma^*$ induces
a functor, also denoted $\delta$, from the fundamental groupoid of $\Gamma$ to
$F(\Sigma)$.  It follows immediately from the definition that if $L\neq \emptyset$, then
$LL\inv = \delta(\pi_1(\Gamma,\iota))$ and so alternatively we can 
describe $L$ as $\delta(\pi_1(\Gamma,\iota))\delta(p)$ where $p$ is any path from $\iota$ to $\tau$.

Next suppose that $\mathscr A$ is finite and assume still that 
$L\neq \emptyset$.  Then $\pi_1(\Gamma,\iota)$ is finitely generated 
and so $LL\inv$ is finitely generated and hence $L$ is finitely generated by Proposition~\ref{generators}.

Conversely, let $X\subseteq F(\Sigma)$ and let $w\in F(\Sigma)$.  Define a literal dual automaton by taking a bouquet of subdivided circles at a base point $\iota$ labeled by the elements of $X$ (with the appropriate dual edges) and attach a thorn labeled by $w$ from $\iota$ to a new vertex $\tau$ (again with the appropriate dual edges).  Then the language recognized by the resulting literal dual automaton is $\langle X\rangle w$ and the automaton is finite if $X$ is finite.
\end{proof}

Notice that the proof of Proposition~\ref{dualautomatalanguages} is effective.
If $G$ is a group generated by $\Sigma$ and $\varphi\colon \til F(\Sigma)\to G$ is
the canonical morphism, then the subset of $G$ \emph{recognized} by the dual
automaton $\mathscr A$ is by definition $\varphi(L(\mathscr A))$.  It follows that
a subset of $G$ is recognized by a finite dual automaton if and only if it
is empty or a coset of a 
finitely generated subgroup. We obtain:

\begin{proposition}\label{generalizedwordproblem}
Let $G$ be a group generated by a finite set $\Sigma$.  Then the following are equivalent:
\begin{enumerate}
\item The uniform generalized word problem is decidable for $G$.
\item Uniform membership is decidable in finitely generated cosets of $G$.
\item Uniform membership is decidable in subsets of $G$ recognized by finite
  dual automata over $\Sigma$.
\item There is an algorithm which given a finite dual automaton $\mathscr A$
  over $\Sigma$ as input, determines whether $1$ belongs to the subset of $G$ 
  recognized by $\mathscr A$.
\item There is an algorithm which given a finite literal dual automaton
  $\mathscr A$ over 
  $\Sigma$ as input, determines whether $1$ belongs to the subset of $G$ recognized by $\mathscr A$.
\item There is an algorithm which given a finitely generated coset $A$ of $G$
  (by a generating set) determines whether $1\in A$.
\end{enumerate}
\end{proposition}
\begin{proof}
The equivalence of the first two items is clear.  The equivalence of 2 and 3 follows from Proposition~\ref{dualautomatalanguages}.  Clearly 3 implies 4 implies 5.  To see that 5 implies 6, suppose $A=\ov X$.  If $X=\emptyset$, there is nothing to prove.  Otherwise, by Proposition~\ref{generators} we can find a generating set for $H=AA\inv$, and $A=Ha$ where $a\in X$.  The proof of Proposition~\ref{dualautomatalanguages} then effectively constructs a literal dual automaton recognizing $A$.  That 6 implies 1 follows from the observation that $g\in H$ if and only if $1\in Hg\inv$ and Proposition~\ref{generators}, which allows us to effectively switch between coset generators and generators of a subgroup.
\end{proof}

\section{The generalized word problem} \label{GWP}

In this section we use dual automata to give a technically 
simpler proof of a result from~\cite{KaWeMy05}
on the decidability of the uniform generalized word problem 
for certain graphs of groups.  We do not obtain algorithmically 
the induced splitting of subgroup, as is done in~\cite{KaWeMy05}.

To make clear the main idea, we first use dual automata to give a 
short proof that the free product of groups with decidable generalized 
word problem again has decidable generalized word problem, 
a result due to Mikhailova \cite{Mih59}.  Although strictly speaking this is a special case of our main result, it seems worth proving separately to isolate the key idea.

\begin{theorem}[Mikhailova~\cite{Mih59}]
Let $G_1,G_2$ be groups with decidable uniform generalized word problem.  Then the free product $G_1\ast G_2$ has decidable uniform generalized word problem.
\end{theorem}
\begin{proof}
Let $\Sigma_1,\Sigma_2$ be disjoint generating sets for $G_1$ and $G_2$ and put $\Sigma=\Sigma_1\cup \Sigma_2$.  Then $\Sigma$ is a generating set for $G=G_1\ast G_2$. By a syllable of a word $w\in \til \Sigma^*$, we mean a maximal non-empty factor of $w$ that can be written over a single alphabet $\til \Sigma_i$.  Let $\varphi\colon F(\Sigma)\to G$ be the projection.

Suppose that $\mathscr A$ is a finite literal dual automaton over $\Sigma$ with
initial vertex $\iota$ and terminal vertex $\tau$.  We perform the following saturation
procedure.  Start with $\mathscr A_0=\mathscr A$.  Assume inductively
$\mathscr A_i$ is obtained from $\mathscr A_{i-1}$ by adding a new edge
labeled by $1$ together with its inverse edge in 
such a way that $\varphi(L(\mathscr A_i))=\varphi(L(\mathscr A_{i-1}))$ 
for $i\geq 1$, but no vertices are added.  Suppose that there is a
pair $p,q$ of distinct vertices with no edge from $p$ to $q$ labeled by $1$
and that, for some $i=1,2$, there is an element of $\Sigma_i^*$ representing
$1$ in $G_i$ accepted by the finite dual automaton over $\Sigma_i$ obtained by
keeping only those edges of $\mathscr A_i$ labeled by elements of $\Sigma_i$
or by $1$, where we take the initial vertex to be $p$ and terminal vertex to be
$q$. Then we add an edge labeled by $1$ from $p$ to $q$, and the corresponding
inverse edge labeled by $1$ from $q$ to $p$, to obtain $\mathscr A_{i+1}$.
Otherwise the algorithm halts. Clearly $\varphi(L(\mathscr A_i)) = \varphi(L(\mathscr
A_{i+1}))$.   This procedure can be done effectively by Proposition~\ref{generalizedwordproblem}
since the
uniform generalized word problem is decidable 
in $G_1$ and $G_2$.  It eventually stops since we add no 
new vertices.  Let $\mathscr B$ be the final automaton obtained when the algorithm terminates.

We claim that $1\in \varphi(L(\mathscr A))$ if and only if $1$ labels an edge from
$\iota$ to $\tau$ in $\mathscr B$.  Since $\varphi(L(\mathscr B))=\varphi(L(\mathscr
A))$, trivially if there is an edge from $\iota$ to $\tau$ in $\mathscr B$
labeled by $1$, then $1\in \varphi(L(\mathscr A))$.  Conversely, suppose $1\in
\varphi(L(\mathscr A))$ and let $w$ be a word accepted by $\mathscr B$ with $\varphi(w)=1$ having a minimum number of syllables.  If $w=1$ or has one syllable, then by construction of $\mathscr B$ there is an edge labeled by $1$ from $\iota$ to $\tau$.  Otherwise, the normal form theorem for free products implies $w$ has a syllable representing $1$ in one of the free factors.  But then by construction of $\mathscr B$, the part of the accepting path for
$w$ traversed by this syllable can be replaced by a single edge labeled by $1$ and so
there is a word with fewer syllables accepted by $\mathscr B$ and mapping to $1$ in $G$.  This contradiction completes the proof. 
% let $w$ be a minimal length word with $\varphi(w)=1$ and
% that is accepted by $\mathscr B$.  If $|w|\geq 1$, then by the normal form
% theorem for free products there must be a factorization $w=uvz$ with $v\in
% \til \Sigma_i^*$ a non-empty word such that $v$ is $1$ in $G_i$, for some $i$.
% But then by construction of $\mathscr B$ (using in particular that each edge is labeled by a letter or $1$), the part of the accepting path for
% $w$ traversed by $v$ can be replaced by a single edge labeled by $1$ and so
% $uz$ is a shorter word accepted by $\mathscr B$ mapping to $1$ in $G$.  This
% contradiction shows $|w|=0$, that is, $w$ is the empty word.  Hence there is a
% path labeled by just $1$ from $\iota$ to $\tau$ in $\mathscr B$.  But by
% construction of $\mathscr B$ this means there 
% is an edge labeled by $1$ from $\iota$ to $\tau$.
\end{proof}

We briefly recall the definition of a graph of groups and its
fundamental group; a detailed introduction can be found in~\cite{Serre03}.
A \textit{graph of groups} $\mathbb G=(G,Y)$  consists of a graph $Y$ and
\begin{itemize}
\item[(i)] for each vertex $v \in V(Y)$, a group $G_v$;
\item[(ii)] for each edge $y \in E(Y)$, a group $G_y$ such that $G_y = G_{y\inv}$;
\item[(ii)] for each edge $y \in E(Y)$, monomorphisms $\alpha_y\colon G_y \to G_{\alpha(y)}$ and
$\omega_y \colon G_y \to G_{\omega(y)}$ such that $\alpha_y = \omega_{y\inv}$ for
all $y \in E(Y)$.
\end{itemize}
We assume that the groups $G_v$ intersect only in the identity, and
that they are disjoint from the edge set $E(Y)$. For each $v \in
V(Y)$, let $\langle \Sigma_v \mid R_v \rangle$ be a presentation
for $G_v$, with the different generating sets $\Sigma_v$ disjoint.  Let $\Delta$ be a 
set containing exactly one edge from each orbit of the involution $y\mapsto y\inv$ on $E(Y)$; we identify $E(Y)$ and $\til \Delta$ when convenient. Let $\Sigma$ be the (disjoint) union of all
the sets $\Sigma_v$ and $\Delta$.   We
define a group $F(G,Y)$ by the presentation
\begin{equation*}
F(G,Y)  =  \langle \Sigma \mid  R_v \ \left(v \in V(Y) \right),
y\omega_y(g)y\inv = \alpha_y(g) \ \left(y \in E(Y), g \in G_y \right) \ \rangle.
\end{equation*}

Fix a vertex $v_0 \in V(Y)$. A word in $w \in \til \Sigma^*$ is of
\textit{cycle type at $v_0$} if it is of the form $w = w_0 y_1 w_1
y_2 w_2 \dots y_n w_n$ where:
\begin{itemize}
\item[(i)]  $y_i \in E(Y)$ for all $1\leq i\leq n$;
\item[(ii)] $y_1\cdots y_n$ is a path in $Y$ starting and ending at $v_0$;
\item[(iii)] $w_0\in \til \Sigma_{v_0}^*$;
\item[(iv)] for $1\leq i\leq n$, $w_i\in \til \Sigma_{\omega(y_i)}^*$.
\end{itemize}
The images in $F(G,Y)$ of the words of cycle type at $v_0$ form a
subgroup $\pi_1(G, Y, v_0)$ of $F(G,Y)$, called the
\textit{fundamental group of $(G,Y)$ at $v_0$}. The fundamental group of a connected graph of groups is (up to isomorphism)
independent of the choice of vertex $v_0$.

Let $\mathbb G=(G,Y)$ be a graph of groups. Let $v_0$ be a 
vertex of $Y$ and fix a spanning tree $T$ for $Y$.  For a 
vertex $v$, let $p_v$ be the unique geodesic path in $T$ 
from $v_0$ to $v$.  The fundamental group $H=\pi_1(G, Y, v_0)$ 
is generated by the words of cycle type at $v_0$ of the form 
$p_{\alpha(y)}yp_{\omega(y)}\inv$ with $y\in \Delta \setminus T$ and  
$p_vxp_v\inv$ with $x\in \Sigma_v$.  In particular, if $Y$ is finite and each of the vertex groups 
is finitely generated, then the fundamental group is finitely generated.

A graph of groups $\mathbb G$ is \emph{benign}~\cite{KaWeMy05} if the following conditions hold:
\begin{enumerate}
\item For each vertex $v\in V(Y)$ and each edge $y\in E(Y)$ with $\omega(y)=v$, there is an algorithm which given a  finitely generated subgroup $K$ of $G_v$ (in terms of a finite generating set for $K$ given by words from $\til \Sigma_v^*$) and an element $g \in G_v$ (via a word in $\til \Sigma_v^*$) determines whether $Kg \cap \omega_y(G_y)$ is empty and if it is non-empty returns an element of the intersection  (represented by a word in the alphabet $\til \Sigma_v$);
\item  Each edge group $G_y$ is Noetherian, or equivalently, all its subgroups are finitely generated;
\item The uniform generalized word problem is decidable for each edge group $G_y$;
\item For each vertex $v\in V(Y)$ and each edge $y\in E(Y)$ with $\omega(y)=v$, there is an algorithm which given a finitely generated subgroup $K$ of $G_v$ (represented by a finite set of words over $\til \Sigma_v^*$ generating $K$) computes a finite generating set for $K\cap \omega_y(G_y)$ as words over $\til \Sigma_v$. 
Note that $K\cap \omega_y(G_y)$ must be finitely generated since $\omega_y(G_y)$ is Noetherian.
\end{enumerate}

It is immediate from Proposition~\ref{generators} that 1 and 4 from the definition of a benign graph of groups are jointy equivalent to the following statement.

\begin{enumerate}
\item [(5)] For each vertex $v\in V(Y)$ and each edge $y\in E(Y)$ with $\omega(y)=v$, there is an algorithm which given a finitely generated coset $A$ of $G_v$ (in terms of a finite generating set for the coset given by words from $\til \Sigma_v^*$) produces a finite generating set for the (possibly empty) coset $A\cap \omega_y(G_y)$ of $\omega_y(G_y)$ (represented by words in the alphabet $\til \Sigma_v$).  Again $A\cap \omega_y(G_y)$ is finitely generated since $G_y$ is Noetherian.
\end{enumerate}

We remark that in a benign graph of groups, given an element of $\omega_y(g)$ as a word $w$ in 
the generators $\til \Sigma_{\omega(y)}$ one can find effectively a word $v$ in the generators $\til \Sigma_{\alpha(y)}$ representing $\alpha_y(g)$.  
To see this assume that $G_y$ is finitely generated by $\Sigma_y$. Then 
the monomorphisms $\alpha_y$ and $\omega_y$ can be represented by 
mappings $\alpha'_y \colon \Sigma_y \to \til \Sigma_{\alpha(y)}^*$,
$\omega'_y \colon \Sigma_y \to \til\Sigma_{\omega(y)}^*$.
Then, given $w \in  \til \Sigma_{\omega(y)}^*$ one enumerates all words
$u \in \til\Sigma_y^*$ until a word $u$ with $\omega'_y(u) = w$ is found.
This word $u$ represents $g \in G_y$ and we can compute $v = \alpha'_y(u)$.
We shall use this fact below without comment.

\begin{theorem}[Kapovich, Weidmann, Myasnikov]\label{thm_gog}
Let $\mathbb G=(G,Y)$ be a finite, connected, non-empty benign graph of finitely
generated groups with underlying graph $Y$. Then the fundamental group
of $\mathbb G$ has decidable uniform generalized word problem if and only if
every vertex group does.
\end{theorem}
\begin{proof}
Let $\varphi\colon F(\Sigma)\to F(G,Y)$ be the projection.  We retain the above notation.  In particular, we continue to use $H$ to denote the fundamental group of $\mathbb G$.
Since each vertex group embeds into the fundamental group~\cite{Serre03},
one implication is immediate.

Suppose that $K$ is a finitely generated subgroup of $H$.  We assume its generators are given as words over $\til \Sigma$ of cycle type at $v_0$.   Let $g\in H$ be given by a word of cycle type at $v_0$ and construct the literal dual automaton $\mathscr A$ over $\Sigma$ from the proof of Proposition~\ref{dualautomatalanguages} recognizing $Kg\inv$ (as a coset of $F(G,Y)$).  Then by construction $g\in K$ if and only if there is a word of cycle type at $v_0$ representing $1$ in $F(G,Y)$ and reading a path from $\iota$ to $\tau$.  We now perform a saturation procedure to $\mathscr A$ to obtain a new dual automaton $\mathscr B$ over $\Sigma$ containing $\mathscr A$  with the same vertex set and the same initial and terminal vertices.  Moreover, $\mathscr B$ will have the property that $\varphi(L(\mathscr A))=\varphi(L(\mathscr B))$ and that $g\in K$ if and only if there is an edge from $\iota$ to $\tau$ in $\mathscr B$ labeled by $1$.

The saturation procedure continues as long as one of the following steps can be performed.   We assume that at each stage of the construction, we have added no new vertices and that all edges are labeled by an element of $\til \Sigma_v^*$, for some $v$, or by an element of $E(Y)$.

 Suppose that the automaton at the current phase of the saturation procedure is $\mathscr A_i$.
If $\Lambda\subseteq \Sigma$ and $p,q$ are vertices of a dual automaton $\mathscr C$ over $\Sigma$, denote by $\mathscr C(\Lambda,p,q)$ the dual automaton consisting of all edges of $\mathscr C$ labeled by elements of $\til \Lambda^*$, taking as initial vertex $p$ and as terminal vertex $q$.

\begin{step}
If there are vertices $p\neq q$ so that $1\in \varphi(\mathscr A_i(\Sigma_v,p,q))\subseteq G_v$, then add an edge $p\xrightarrow{1}q$ and an inverse edge $q\xrightarrow{1}p$ if there are not already such edges.  This step can be done effectively because of our assumption that $G_v$ has a decidable uniform generalized word problem.
\end{step}

\begin{step}
Let $p\xrightarrow{y}q$ and $p'\xrightarrow{y}q'$ be edges in $\mathscr A_i$
with $y\in E(Y)$ (not necessarily distinct).  Let $L$ be the coset
$\varphi(\mathscr A_i(\Sigma_{\omega(y)},q,q'))$ and compute, using that the graph
of groups is benign, a finite  generating set $X$ for the (possibly empty)
coset $L\cap \omega_y(G_y)$ represented by words over $\til \Sigma_{\omega(y)}$.   For
each $x\in X$, find a word $w_x\in \til \Sigma_{\alpha(y)}^*$ representing
$\alpha_y\omega_y\inv(x)$ and add an edge $p\xrightarrow{w_x} p'$ and an inverse 
edge $p'\xrightarrow{w_x\inv} p$ if $\varphi(w)\notin \varphi(\mathscr A_i(\Sigma_{\alpha(y)},p,p'))$
(the latter can be checked effectively, 
since $G_{\alpha(y)}$ has a decidable uniform generalized word problem).
\end{step}

This procedure is continued until none of the steps can be performed further.  
Clearly this procedure does not change the accepted subset of $F(G,Y)$, which is the coset $Kg\inv \subseteq H$.

We must show that our procedure stops.  Step 1 can only be performed finitely many
times since we are adding no new vertices.  Since the edge groups are
Noetherian, they satisfy ascending chain condition on cosets.  
Step 2 can only be applied if the coset $\varphi(\mathscr A_i(\Sigma_{\alpha(y)},p,p'))\cap \alpha(G_y)$ is made into a bigger coset  
$\varphi(\mathscr A_{i+1}(\Sigma_{\alpha(y)},p,p'))\cap \alpha(G_y)$ by adding the elements of 
$X$ (following the notation of Step 2) written in the alphabet
$\Sigma_{\alpha(y)}$.   
Thus Step 2 can only be applied a finite number of times.

Hence the procedure eventually terminates with a `saturated' dual automaton $\mathscr B$. The following lemma is crucial.

\begin{lemma}\label{crucial}
Suppose that $p\xrightarrow{y}q$ and $p'\xrightarrow{y}q'$ are edges in $\mathscr B$ and that 
$w\in \til \Sigma_{\omega(y)}^*$ labels a path from $q$ to $q'$ in $\mathscr B$ and satisfies $\varphi(w)=\omega_y(g)$.  
Then there is a word $u\in \til \Sigma_{\alpha(y)}^*$ labeling a path from $p$ to $p'$ in $\mathscr B$ with $\varphi(u) = \alpha_y(g)$.
\end{lemma}

\begin{proof}
The element $\varphi(w)=\omega_y(g)$ belongs to the coset \[C=\varphi(\mathscr B(\Sigma_{\omega(y)},q,q'))\cap \omega_y(G_y).\]  
By saturation of $\mathscr B$ under Step 2, $\varphi(\mathscr B(\Sigma_{\alpha(y)},p,p'))\cap \alpha_y(G_y)$ 
contains a generating set of the coset $\alpha_y(\omega_y\inv(C))$ (represented by words over the alphabet $\til \Sigma_{\alpha(y)}$).  
Since in fact $\varphi(\mathscr B(\Sigma_{\alpha(y)},p,p'))\cap \alpha_y(G_y)$ is a coset,  
it thus contains $\alpha_y(\omega_y\inv(C))$ and so there is a word 
$u\in \til \Sigma_{\alpha(y)}^*$ such that $\varphi(u)=\alpha_y(g)$ and $u$ labels a path in $\mathscr B$ from $p$ to $p'$, as required.
\end{proof}

We now claim that $g\in K$ if and only if $1$ labels an edge from $\iota$ to
$\tau$ in $\mathscr B$.  If there is an edge from $\iota$ to $\tau$ labeled by
$1$, then trivially $g\in K$.  Conversely, if $g\in K$ then 
there is a word $w$ of cycle type at $v_0$ reading from $\iota$ 
to $\tau$ in $\mathscr A$ (and hence in $\mathscr B$) with $\varphi(w)=1$.  Choose $w = w_0 y_1 w_1
y_2 w_2 \dots y_n w_n$ satisfying (i)--(iv) with $n$ minimal so 
that $\varphi(w)=1$ and $w$ labels a path from $\iota$ to $\tau$ in $\mathscr B$.  
If $n=0$, then saturation under Step 1 shows that $1$ labels an 
edge from $\iota$ to $\tau$. Suppose $n\geq 1$.  We obtain a contradiction.  It follows by
\cite[Theorem~I.11]{Serre03} that there exists $i$ such that
$y_{i+1} =y_{i}\inv$ and $w_i$ represents an element $\omega_{y_i}(g_i)$ for
some $g_i \in G_{y_i}$. It now
follows from Lemma~\ref{crucial} that we can replace the factor $y_iw_iy_{i+1}$ by a word
$w_i'\in \til \Sigma_{\alpha(y_i)}^*=\til \Sigma_{\omega(y_{i-1})}^* = \til
\Sigma_{\alpha(y_{i+2})}^* $ to obtain a new word $w_0y_1\cdots
y_{i-1}w_{i-1}w_i'w_{i+1}y_{i+2}\cdots w_n$ of cycle type at $v_0$ labeling a
path in $\mathscr B$ from $\iota$ 
to $\tau$ and mapping to $1$ in $F(G,Y)$, again a contradiction to minimality of $n$.  This completes the proof.
\end{proof}

\def\cprime{$'$}

%\bibliographystyle{abbrv}
%\bibliography{bib}
%\bibliography{../../../BIBTEX/bib}

\begin{thebibliography}{10}

\bibitem{AvWi89}
J.~Avenhaus and D.~Wi{\ss}mann.
\newblock Using rewriting techniques to solve the generalized word problem in
  polycyclic groups.
\newblock In {\em Proceedings of the ACM-SIGSAM 1989 International Symposium on
  Symbolic and Algebraic Computation}, pages 322--337. ACM Press, 1989.

\bibitem{Benois69}
M.~Benois.
\newblock Parties rationnelles du groupe libre.
\newblock {\em C. R. Acad. Sci. Paris}, S{\'{e}}r. A, 269:1188--1190, 1969.

\bibitem{BenSak86}
M.~Benois and J.~Sakarovitch.
\newblock On the complexity of some extended word problems defined by
  cancellation rules.
\newblock {\em Information Processing Letters}, 23(6):281--287, 1986.

\bibitem{KaSiSt06}
M.~Kambites, P.~V. Silva, and B.~Steinberg.
\newblock On the rational subset problem for groups.
\newblock {\em Journal of Algebra}, 309(2):622--639, 2007.

\bibitem{KaWeMy05}
I.~Kapovich, R.~Weidmann, and A.~Myasnikov.
\newblock Foldings, graphs of groups and the membership problem.
\newblock {\em International Journal of Algebra and Computation},
  15(1):95--128, 2005.

\bibitem{LohSte08}
M.~Lohrey and B.~Steinberg.
\newblock The submonoid and rational subset membership problems for graph
  groups.
\newblock {\em Journal of Algebra}, 320(2):728--755, 2008.

\bibitem{Mal83}
A.~I. Malcev.
\newblock On homomorphisms onto finite groups.
\newblock {\em American Mathematical Society Translations, Series 2},
  119:67--79, 1983.
\newblock Translation from Ivanov. Gos. Ped. Inst. Ucen. Zap. 18 (1958) 49--60.

\bibitem{Mih59}
K.~A. Miha{\u\i}lova.
\newblock The occurrence problem for free products of groups.
\newblock {\em Doklady Akademii Nauk SSSR}, 127:746--748, 1959.

\bibitem{Mih66}
K.~A. Miha{\u\i}lova.
\newblock The occurrence problem for direct products of groups.
\newblock {\em Math. USSR Sbornik}, 70:241--251, 1966.
\newblock English translation.

\bibitem{Rip82b}
E.~Rips.
\newblock Subgroups of small cancellation groups.
\newblock {\em Bulletin of the London Mathematical Society}, 14:45--47, 1982.

\bibitem{Rom74}
N.~S. Romanovski{\u\i}.
\newblock Some algorithmic problems for solvable groups.
\newblock {\em Algebra i Logika}, 13(1):26--34, 1974.

\bibitem{Rom80}
N.~S. Romanovski{\u\i}.
\newblock The occurrence problem for extensions of abelian groups by nilpotent
  groups.
\newblock {\em Sibirsk. Mat. Zh.}, 21:170--174, 1980.

\bibitem{Schein66}
B.~Schein.
\newblock Semigroups of strong subsets.
\newblock {\em Volzhsky Matematichesky}, 4:180--186, 1966.

\bibitem{Serre03}
J.-P. Serre.
\newblock {\em Trees}.
\newblock Springer, 2003.

\bibitem{Stal83}
J.~R. Stallings.
\newblock Topology of finite graphs.
\newblock {\em Invent. Math.}, 71(3):551--565, 1983.

\bibitem{Umi95}
U.~U. Umirbaev.
\newblock The occurrence problem for free solvable groups.
\newblock {\em Sibirski\u\i\ Fond Algebry i Logiki. Algebra i Logika},
  34(2):211--232, 243, 1995.

\bibitem{Wise03}
D.~T. Wise.
\newblock A residually finite version of {R}ips's construction.
\newblock {\em Bulletin of the London Mathematical Society}, 35(1):23--29,
  2003.

\end{thebibliography}

\end{document}